\documentclass[10pt,leqno]{article} %%%%leqno es para numeraci\'{o}n de
   \usepackage[centertags]{amsmath}
   \usepackage{amsfonts}
   \usepackage{amsmath}
   \usepackage{amssymb}
   \usepackage{amsthm}
   \usepackage{newlfont}
   \usepackage[latin1]{inputenc}%%%%para que reconozca las tildes.
   \newtheorem{thm}{Theorem}[section]

   \font\msbmnormal=msbm10 \font\msbmpeq=msbm7 \font\msbmmuypeq=msbm5
   \newfam\numeros
   \textfont\numeros=\msbmnormal \scriptfont\numeros=\msbmpeq
   \scriptscriptfont\numeros=\msbmmuypeq

   \numberwithin{equation}{section} %%%%%%numera las ecuaciones (1.1),
   \title{Matrix differential equations and scalar polynomials satisfying higher order
   recursions.
   \footnote{The
      work of the first author is  partially supported by  D.G.E.S,
      ref. MTM-2006-13000-C03-01, FQM-262 (\textit {Junta de
   Andaluc\'{\i}a}), that of the second
   author is partially supported by NSF grant \#DMS 0204682.}}
   \author{Antonio J. Dur\'{a}n$^{\dagger}$ and  F. Alberto
   Gr\"unbaum$^{\ddagger}$\\
       \footnotesize $\dagger$  \footnotesize
        \  Departamento de An\'{a}lisis Matem\'{a}tico.
       Universidad de Sevilla \\
       \footnotesize Apdo (P. O. BOX) 1160. 41080 Sevilla. Spain.
   duran@us.es \\
       $\ddagger$ \footnotesize\ Department of Mathematics. University of
   California,  Berkeley \\
        \footnotesize Berkeley,CA 94720  U.S.A. grunbaum@math.berkeley.edu
   \\
   \ \ }
   \date{}
   
\begin{document}
   \maketitle

   \begin{abstract}

We show that any scalar differential operator with a family of polynomials
as its common eigenfunctions leads canonically to a matrix differential operator
with the same property. The construction of the corresponding family of matrix valued polynomials
has been studied in \cite{D1,D2,DV} but the existence of a differential
operator having them as common eigenfunctions had not been considered 
This correspondence goes only one way and most matrix valued situations do not arise in this fashion.

We illustrate this general construction with a few examples.
In the case of some families of scalar valued polynomials introduced in \cite{GH} we
take a first look at the algebra of all matrix differential operators that share
these common eigenfunctions and uncover a number of phenomena that are new to the matrix valued case.

    \end{abstract}

\bigskip

\section{Introduction}

\bigskip

The last few years have witnessed some progress in the problem of finding
explicit families of matrix valued orthogonal polynomials that are joint
eigenfunctions of some fixed differential operator with matrix coefficients,
and a (possibly matrix valued) eigenvalue that depends on the degree of the
polynomial. Most of these examples involve differential operators of order
two, an issue raised in the case of symmetric operators in \cite{D3}.
The subject of matrix valued orthogonal polynomials, without any reference to
differential operators, was initiated in two papers by
M.G. Krein, \cite {K1,K2}. The search of situations where these polynomials
satisfy extra properties, such as the one above, makes it more likely that they will be useful in applications.

There are by now four main ways to search for cases with this particular extra property:
appealing to some group representation structure, as in
\cite{GPT1,GPT2,GPT3,GPT4,PT}, solving (in the case of symmetric differential operators) an appropriate set of
differential equations as in \cite{DG1,DG2,Du4,G,GPT3} or a set of moment equations as in \cite{D3}, and  finally
and so far much less succesfully, solving the so called {\em
ad-conditions} as in \cite{CG3}. One can see that these ad-conditions
are necessary and sufficient for the existence of a differential
operator,
 but finding all solutions remains a huge task. Most of the examples obtained by this route correspond most likely to an "orthogonality functional" that is not
given by a positive definite matrix as envisaged by M. G. Krein. 
It appears reasonable at this point to consider in some detail any method that gives explicit examples.

\bigskip

One such path is taken in \cite {CG1}
where one finds instances where there is a differential operator of order one,
and in \cite{CG2}
where one explores the algebra of all differential operators going along
with a fixed family of differential operators.

\bigskip

In this paper we take the initial steps in yet another path. This comes about by putting together
two previously unrelated lines of work.

The first one comes from studying
the relation between scalar valued polynomials satisfying higher order
recursion relations and matrix valued polynomials.
This has been considered in \cite{D1,D2,DV}, and is reviewed here in section 2. These papers do not look into
the issue of possible differential equations satisfied by these polynomials.

The second one comes from an extension of the
well known Krall-Laguerre and Krall-Jacobi  orthogonal polynomials.
This has been considered 
in \cite{GH} to produce, via Darboux factorization, families of (scalar) polynomials
that (in general) satisfy five term recursion relations and fourth order
differential operators with polynomial coefficients. The relevant parts of this
construction are recalled in section 4.

\bigskip

These two tools are combined in this paper as follows:
we first establish a very general result that produces one matrix valued differential
operator every time that we start with a scalar one (but not necessarily the other way around, as spelled out later).
This general result is first illustrated in the
case of the Laguerre polynomials where the orders of all operators is equal to two.
This completes the material in section 3.
We then illustrate, in section 4, the case of higher order operators by looking at
the extended scalar valued  Krall-Laguerre and Krall-Jacobi
polynomials.

\bigskip

A sligthly more detalied description of the contents of sections 3 and 4 follows.
Given specific families of scalar valued polynomials
we consider the corresponding family of matrix valued polynomials as in \cite{D1,D2,DV}.
We find that in the Krall-Laguerre case there is (up to scalars)
only one differential operator with matrix valued coefficients (and of order four) going along with each one
of these examples. In particular, there is no lower order operator with this property.
By going up to
order six we see that the corresponding algebra cannot be generated by the fourth order operator constructed in the
canonical fashion considered in the first part of the paper.
The case of the extended Krall-Jacobi polynomials leads to a richer variety of examples. In particular we get a third order differential operator in the matrix
case which cannot yield a third order scalar differential operator.
All of this is discussed in several subsections of section 4.

\bigskip

To better compare the results of this paper with previously known material notice that
the situations uncovered so far feature an algebra that includes differential operators of low order,
typically one or two. There are even situations where one finds a nontrivial (i.e. not a scalar multiple of the
identity)
differential operator of order zero.
One could wonder as to the existence of situations where
the lowest order operators in the algebra
needs to be of order higher than two.
The group theoretical
line of attack mentioned above,
is not likely to produce
examples of this kind, since in this case there is a Laplace-Beltrami operator around. Solving the
appropriate differential equations for the higher
order case appears to be a difficult task, while solving the corresponding ad-conditions
appears entirely hopeless at this point.

As a by-product of the results in this paper
we produce
examples of the type envisaged above, which as far as we know have not been seen before.

\bigskip

Another new result is the observation that we have here
a one way street; while the results in \cite{D1,D2,DV}
establish a reversible link between
scalar valued and matrix valued polynomials we see here that if the scalar family is made up of common eigenfunctions
of some differential operators this property extends to the matrix valued polynomials but not (necessarily) the other
way round.

\bigskip

       We close this introduction with some remarks, aimed at a broad audience, about some of the
       possible uses of matrix valued orthogonal polynomials.
    
       One of the many applications of the theory of scalar orthogonal polynomials is to a detailed study
       of a special kind of Markov chains with state space given by the non-negative integers and with a
       one step transition probability matrix that allows only moves to the nearest neighbours. The basic
       paper here is \cite{KMcG}. These authors derive an expression (based on the spectral analysis of the
       corresponding difference operator) for the n-step transition probability between states $i$ and $j$
       in terms of the orthogonal polynomials going along with the one step transition probability matrix.
       There are several advantages to such a formula. For instance if one knows the orthogonality measure
       one can compute the probability of going from states $i,j$ with $i \leq j $ in any number
       of steps from the knowledge of the first $j$ rows of the one step transition matrix. Other uses of
       this connection with orthogonal polynomials lie a bit deeper: by using
       the Stieltjes transform of the orthogonality measure one can see that (in principle) the knowledge
       of the probability of going from state $0$ to state $0$ in an arbitrary number of steps 
       determines the full n-step transition probability matrix for any $n$.

\bigskip

       In the last section of the paper by Karlin and McGregor mentioned above, these authors consider the
       problem of using their methods , based on a three term recursion relation that is semi-infinite, to
       the case of a doubly infinite three term recursion. Although they stop short of introducing 
       explicitly the notion of matrix valued orthogonal polynomials they compute in detail the orthogonality
       matrix measure that is relevant to this problem. The full blown analysis of the problem of a random
       walk (with nearest neighbour transitions) on the set of all integers that was the concern of the
       last section in \cite{KMcG} has been done in detail later in \cite{DRSZ}, \cite{G1}. This extension from
       the study by means of scalar valued orthogonal polynomials
       of birth-and-death processes to the study of so called Quasi-birth-and-death
       processes by means of matrix valued orthogonal polynomials is already a concrete example of a good
       motivation for this study.

       For a situation where the matrix valued orthogonal polynomials happen to satisfy matrix valued
       differential equations which along with the orthogonality measure are an outgrowth of the work started in \cite{GPT1}, 
       see \cite{GI}.

\bigskip

       The best known families of scalar valued orthogonal polynomials, connected with
       the names of 
 Hermite, Laguerre, 
and
Jacobi owe their importance to the differential equations they satisfy and thus they
appear in connection with simple situations involving the Schr\"odinger equation.
In the scalar case the only
possible examples of orthogonal polynomials with a differential operator floating around  are the ones mentioned above. 
On the other hand the situation in the matrix valued case is much more complicated
and it opens the door to an embarrassment
of
riches in terms of examples. The study carried out here is an effort to produce more examples of this novel situation
as a step towards other applications.

\section{Scalar and matrix valued polynomials}

The relation between scalar polynomials satisfying higher order recursions
and matrix valued ones satisfying three term recursions was found by one of us (see
\cite{D1} and \cite{D2}) and it is well described in \cite{DV}.

The procedure shows a way to produce a vector polynomial from
an scalar one $p(x)$ by considering the decomposition of $p(x)$
congruent with the residues of the exponent of $x$ in $p(x)$ modulo $N$.

This is done by using the operators $R_{N,m}$, $m=0,\cdots , N-1$, defined by
\begin{align}\label{operadoresR}
R_{N,m}(p)(x) = \sum_n \frac {p^{(nN+m)}(0)}{(nN+m)!} x^n,
\end{align}
i.e., the operator $R_{N,m}$ takes from $p$ those powers congruent with $m$ modulo
$N$, then removes the common factor $x^m$ and changes $x^N$ to $x$.

From a sequence of scalar polynomials $(p_n)_n$
we can then define a sequence of matrix valued polynomials $P_n(x)$, $n=0,1,\cdots $, by the recipe
\begin{equation} \label{matrixpol}
P_n(x) = \begin{pmatrix}
R_{N,0}(p_{nN})(x) &\dots &R_{N,N-1}(p_{nN})(x) \\
R_{N,0}(p_{nN+1})(x) &\dots &R_{N,N-1}(p_{nN+1})(x) \\
\vdots & &\vdots \\
R_{N,0}(p_{nN+N-1})(x) &\dots &R_{N,N-1}(p_{nN+N-1})(x)
\end{pmatrix}\ .
\end{equation}

This process can be reversed since for any polynomial $p(x)$ we have

$$ p(x)= R_{N,0}(p)(x^N) + xR_{N,1}(p)(x^N)+...+x^{N-1}R_{N,N-1}(p)(x^N)$$

\noindent

It turns out that if
$p_n(x)$, $n = 0,1,2,\dots $, satisfy the
$(2N+1)$-term recurrence relation
\begin{equation}\label{hor}
x^Np_n(x) = c_{n,0}p_n(x) + \sum_{k=1}^N [ c_{n,-k}p_{n-k}(x) +  c_{n,k}p_{n+k}(x)],
\end{equation}
($p_n(x)$ vanishes if $n$ is negative), then the matrix valued ones $P_n$,
$n = 0,1,2,\dots $, satisfy the following three term matrix recurrence relation

\begin{equation}\label{ttrr}
xP_n(x)=A_nP_{n+1}(x)+B_nP_n(x)+C_nP_{n-1}(x)
\end{equation}

where the matrix coefficients are nothing but the $N\times N$
blocks of the $(2N+1)$ banded matrix associated to the scalar family $(p_n)_n$
 featuring in the higher order
recursion that defines the $p_n(x)$. This semi-infinite banded matrix  has rows that look as follows

\bigskip
first row

$$  c_{0,0}, c_{0,1},\cdots , c_{0,N},0,0,0,\cdots $$

second row

$$  c_{1,-1}, c_{1,0}, c_{1,1}, c_{1,2}, \cdots ,c_{1,N},0,0,0,\cdots $$

third row

$$  c_{2,-2}, c_{2,-1}, c_{2,0}, c_{2,1}, c_{2,2}, c_{2,3}, \cdots , c_{2,N}$$

\bigskip

and for i large enough, the ${(i+1)}_{th}$ row looks like

$$  0,0,\cdots ,0,c_{i,-N},c_{i,-N+1},\cdots ,c_{i,-1},c_{i,0},c_{i,1},\cdots ,c_{i,N},0,0,\cdots $$

Given a block-tridiagonal matrix (that defines, by using (\ref{ttrr}), a sequence of matrix
valued polynomials) one can always assume, by using an $n$ dependent unitary matrix, that the off-diagonal blocks are
triangular, and thus the matrix becomes a banded matrix with scalar entries. This shows that
using the operators $R_{N,m}$, $m=0,\cdots ,N-1$, one can also go from matrix valued polynomials
satisfying a three term recurrence relation like (\ref{ttrr}) to scalar valued polynomials satisfying
a higher recurrence relation as in (\ref{hor}).

\bigskip

The polynomial $x^N$ in the formula (\ref{hor}) does not play any special role. One can replace it by $(x-a)^N$ for an appropriate choice of the constant $a$;
in this case, one has to use the operators
$$
R_{N,m,a}(p)(x) = \sum_n \frac {p^{(nN+m)}(a)}{(nN+m)!} x^n,
$$
instead of the $R_{N,m}$ defined in (\ref{operadoresR}) (see \cite{DV}, pp. 269, for the general case when in (\ref{hor}),
the polynomial $x^N$ is substituted by a generic polynomial $h(x)$).

\bigskip

We say that a higher order recursion (\ref{hor}) for $(p_n)_n$ is symmetric if $c_{n,-k}=\bar c_{n-k,k}$,
and, analogously, that the matrix three term recurrence relation (\ref{ttrr}) is symmetric if
$A_n=C_{n+1}^*$. As it was shown in \cite{DV},
a symmetric higher order recursion gives a symmetric matrix three term recurrence relation.
\bigskip

The symmetric case  corresponds to a Hermitian $(2N+1)$-banded matrix and  the associated matrix polynomials
defined by (\ref{matrixpol}) are then orthonormal with respect to a weight matrix (i.e. a positive
definite matrix of measures having finite moments of any order).

Given a higher order scalar recursion that is not symmetric there are two ways to try to
symmetrize it. Starting from the polynomials $p_n(x)$ defined by the recursion one
can define new polynomials
$r_n(x)=\gamma_np_n(x)$, and look for a convenient choice of the sequence of numbers
$(\gamma _n)_n$ so that the corresponding higher order recursion for $(r_n)_n$ is symmetric.

A different, and more general approach consists of producing from the scalar polynomials $p_n(x)$ the 
matrix polynomials $P_n(x)$ as above. Now one considers 
a new sequence of matrix polynomials $R_n(x)=\Gamma _nP_n(x)$, and looks
for a convenient choice of the sequence of matrices $(\Gamma_n)_n$ so that the matrix three term recurrence relation for $(R_n)_n$ is
symmetric. It is easy to see that the first attempt is a special case of the second one, which consists of insisting that the matrices $\Gamma_n$ be diagonal.

\section{Producing matrix differential equations out of scalar ones}

Our starting point is a sequence $(p_n)_n$ of polynomials, with deg $p_n=n$,
satisfying the differential equation
\begin{equation}\label{ede}
\sum_{l=1}^ma_l(x)p_n^{(l)}(x)=\gamma _np_n(x),
\end{equation}
where $a_l$, $l=1,\cdots, m$, are polynomials of degree not bigger than
$l$ and $\gamma _n$ are real numbers.

%\section{Producing matrix differential equations out of scalar ones}

%Our starting point is a sequence $(p_n)_n$ of polynomials, with deg $p_n=n$,
%satisfying the differential equation
%\begin{equation}\label{ede}
%\sum_{l=1}^ma_l(x)p_n^{(l)}(x)=\gamma _np_n(x),
%\end{equation}
%where $a_l$, $l=1,\cdots, m$, are polynomials of degree not bigger than
%$l$ and $\gamma _n$ are real numbers.

We split up the polynomials $p_n$ as explained
in the previous section to get a sequence of matrix polynomials $P_n$ of size $N\times N$.
We now show that these matrix polynomials inherit a matrix differential equation
from the differential equation (\ref{ede}) for the scalar polynomials $p_n$.
The proof is constructive, so that it produces an explicit expression
for the differential coefficients $A_k$, $k=0,\cdots , m$, in the new differential operator.

For the sake of simplicity, we show the result in full only for the case $N=2$ (the general
case can be proved in the same way).

\begin{thm}\label{teo2.1}
Assume that $(p_n)_n$ is a sequence of scalar polynomials satisfying a  differential equation
of the form (\ref{ede})
where $a_l$, $l=1,\cdots, m$, are polynomials of degree not bigger than
$l$ and $\gamma _n$ are numbers.
Then the matrix polynomials $P_n$ defined from $p_n$ by (\ref{matrixpol}) satisfy
the  differential equation
\begin{align*}
\sum_{k=0}^mP_n^{(k)}(x)A_k(x)=\Gamma _nP_n(x),
\end{align*}
where the coefficients $A_k$, $k=0,\cdots , m,$ are given by
\begin{align*}
A_k&=\left( \sum_{l=k}^mC_{k,l}\right),\\
C_{k,l}(x)&=\begin{pmatrix}
b_{k,l,0}(x)&b_{k,l,1}(x)\\ xb_{k,l,1}(x)+lb_{k,l-1,0}(x)&
b_{k,l,0}(x)+lb_{k,l-1,1}(x)
\end{pmatrix}\begin{pmatrix}a_{l,0}(x)&a_{l,1}(x)\\xa_{l,1}(x)&a_{l,0}(x)\end{pmatrix},\\
b_{k,l}&=\left(\frac{(-1)^k}{k!}\sum_{j=[(l+1)/2]}^k(-1)^j{k\choose j}2j(2j-1)\cdots (2j-l+1)\right)x^{2k-l},\\
\end{align*}
(as usual $[x]$ denotes the integer part of $x$)
and we write $p_0=R_{2,0}(p)$ and $p_1=R_{2,1}(p)$, where $R_{2,0}$
and $R_{2,1}$ are the operators defined in (\ref{operadoresR}) for $N=2$.

The eigenvalues $\Gamma_n$, $n\ge 0$, are given by
$$
\Gamma_n=\begin{pmatrix}\gamma_{nN}&0&\cdots &0\\
0&\gamma_{nN+1}&\cdots &0\\\vdots&\vdots&\ddots&\vdots \\
0&0&\cdots&\gamma_{nN+N-1}\end{pmatrix}.
$$
\end{thm}

\begin{proof}

To simplify the
expressions, we again write $p_0=R_{2,0}(p)$ and $p_1=R_{2,1}(p)$, so that $p(x)=p_0(x^2)+xp_1(x^2)$.

A simple computation shows that
$$
(f(x^2))^{(l)}=\sum _{k=[(l+1)/2]}^lb_{k,l}(x)f^{(k)}(x^2),
$$
where
$$
b_{k,l}(x)=\left(\frac{(-1)^k}{k!}\sum_{j=[(l+1)/2]}^k(-1)^j{k\choose j}
2j(2j-1)\cdots (2j-l+1)\right)x^{2k-l},\quad k>l/2.
$$

Applying this formula to $p_n(x)=p_{n,0}(x^2)+xp_{n,1}(x^2)$, we get after some simplifications
that
\begin{align*}
(p_{n,0}&(x^2)+xp_{n,1}(x^2))^{(l)}\\&=\sum_{k=0}^l\left( b_{k,l,0}(x^2)p_{n,0}^{(k)}(x^2)
+\left( x^2b_{k,l,1}(x^2)+lb_{k,l-1,0}(x^2)\right)p_{n,1}^{(k)}(x^2)\right)
\\ &\quad \quad +
x\left(\sum_{k=0}^l\left( b_{k,l,1}(x^2)p_{n,0}^{(k)}(x^2)
+\left( b_{k,l,0}(x^2)+lb_{k,l-1,1}(x^2)\right)p_{n,1}^{(k)}(x^2)\right)\right),
\end{align*}
where we take $b_{k,k'}=0$, when $k>k'$. This can be written in the form:
\begin{align}\label{secc2.1}
(p_{n,0}&(x^2)+xp_{n,1}(x^2))^{(l)}\\&=\nonumber \sum_{k=0}^l \begin{pmatrix}p_{n,0}^{(k)}(x^2)& p_{n,1}^{(k)}(x^2)\end{pmatrix}
\begin{pmatrix}
b_{k,l,0}(x^2)\\ \left( x^2b_{k,l,1}(x^2)+lb_{k,l-1,0}(x^2)\right)
\end{pmatrix}
\\ & \nonumber  \\
& \nonumber \quad \quad +x\sum_{k=0}^l \begin{pmatrix}p_{n,0}^{(k)}(x^2)& p_{n,1}^{(k)}(x^2)\end{pmatrix}
\begin{pmatrix}
b_{k,l,1}(x^2)\\\left( b_{k,l,0}(x^2)+lb_{k,l-1,1}(x^2)\right)
\end{pmatrix}.
\end{align}
If we write
$$
C_{k,l}(x)=\begin{pmatrix}
b_{k,l,0}(x)&b_{k,l,1}(x)\\ xb_{k,l,1}(x)+lb_{k,l-1,0}(x)&
b_{k,l,0}(x)+lb_{k,l-1,1}(x)
\end{pmatrix}\begin{pmatrix}a_{l,0}(x)&a_{l,1}(x)\\xa_{l,1}(x)&a_{l,0}(x)\end{pmatrix},
$$
and take into account that for any polynomials $p,q$ one has $(pq)_0=p_0q_0+xp_1q_1$ and
$(pq)_1=p_0q_1+p_1q_0$, we get from (\ref{secc2.1}) that
\begin{equation}\label{secc2.2}
\begin{pmatrix}\left(a_l(x)p_n^{(l)}(x)\right)_0
\left(a_l(x)p_n^{(l)}(x)\right)_1\end{pmatrix} =\sum_{k=0}^l
\begin{pmatrix}p_{n,0}^{(k)}(x)& p_{n,1}^{(k)}(x)\end{pmatrix}C_{k,l}.
\end{equation}

From here, it follows that the differential equation (\ref{ede}) gives the following matrix equation for
the vector $\begin{pmatrix}p_{n,0}^{(k)}(x)& p_{n,1}^{(k)}(x)\end{pmatrix}$:
$$
\sum _{k=0}^m \begin{pmatrix}p_{n,0}^{(k)}(x)& p_{n,1}^{(k)}(x)\end{pmatrix}\left( \sum_{l=k}^mC_{k,l}\right)=
\gamma_n\begin{pmatrix}p_{n,0}(x)& p_{n,1}(x)\end{pmatrix}.
$$
It is now easy to finish the proof.

\end{proof}

We illustrate the general construction above with an example.

Take the Laguerre polynomials $(L_n^\alpha)_n$, which satisfy the second order differential equation
$$
x(L_n^\alpha (x))''+(\alpha +1-x)(L_n^\alpha (x))'=-nL_n^\alpha (x).
$$
The proof of Theorem \ref{teo2.1} gives for the matrix polynomials
$$
L^\alpha_{n,2\times 2}=\begin{pmatrix}\displaystyle \frac{L_{2n}^\alpha (\sqrt x)+L_{2n}^\alpha (-\sqrt x)}{2}&
\displaystyle \frac{L_{2n}^\alpha (\sqrt x)-L_{2n}^\alpha (-\sqrt x)}{2\sqrt x}\\
\displaystyle \frac{L_{2n+1}^\alpha (\sqrt x)+L_{2n+1}^\alpha (-\sqrt x)}{2}&
\displaystyle \frac{L_{2n+1}^\alpha (\sqrt x)-L_{2n+1}^\alpha (-\sqrt x)}{2\sqrt x}
\end{pmatrix}
$$
the matrix differential equation
\begin{align*}
\left(L^\alpha_{n,2\times 2}\right) ''&\begin{pmatrix}0&4x\\4x^2&0\end{pmatrix}+
\left(L^\alpha_{n,2\times 2}\right) '\begin{pmatrix}-2x&2\alpha +4\\(8+2\alpha)x&-2x\end{pmatrix}+
L^\alpha_{n,2\times 2}\begin{pmatrix}0&0\\\alpha +1&-1\end{pmatrix}\\
&\quad \quad\quad\quad =
\begin{pmatrix}-2n&0\\0&-2n-1\end{pmatrix} L^\alpha_{n,2\times 2}.
\end{align*}
In this case, the sequence $(L^\alpha_{n,2\times 2})_n$ is orthogonal with respect to the weight
matrix
$$
e^{-\sqrt x}\begin{pmatrix}\displaystyle \frac{1}{\sqrt x} &1\\
1&\sqrt x
\end{pmatrix} .
$$
For any size $N$, the matrix polynomials $(L^\alpha_{n,N\times N})_n$ defined from the Laguerre polynomials
by (\ref{matrixpol}) satisfy the matrix second order differential equation
$$
\left(L^\alpha_{n,N\times N}\right) ''A_2(x)+
\left(L^\alpha_{n,N\times N}\right) 'A_1(x)+
L^\alpha_{n,2\times 2}A_0=
\begin{pmatrix}\gamma_{nN}&0&\cdots &0\\
0&\gamma_{nN+1}&\cdots &0\\\vdots&\vdots&\ddots&\vdots \\
0&0&\cdots&\gamma_{nN+N-1}\end{pmatrix} L^\alpha_{n,2\times 2},
$$
where
$$
A_2(x)=\begin{pmatrix}0&0&0&\cdots &0&4\\
4x^2&0&0&\cdots &0&0\\0&4x^2&0&\cdots &0&0\\
\vdots&\vdots&\vdots&\ddots&\vdots&\vdots \\
0&0&0&\cdots&0&0\\
0&0&0&\cdots&4x^2&0\end{pmatrix},
$$
$$
A_1(x)=\begin{pmatrix}-x&0&0&\cdots &0&4+\alpha /2\\
(N+2+\alpha)x&-x&0&\cdots &0&0\\0&(N+4+\alpha)x&-x&\cdots &0&0\\
\vdots&\vdots&\vdots&\ddots&\vdots&\vdots \\
0&0&0&\cdots&-x&\\
0&0&0&\cdots&(N+2(N-1)+\alpha)x&-x\end{pmatrix},
$$
and
$$
A_0(x)=\begin{pmatrix}N-1/N^2&0&0&\cdots &0&0\\
\alpha+1/N^2&N-2/N^2&0&\cdots &0&0\\0&2\alpha+2^2/N^2&N-3/N^2&\cdots &0&0\\
\vdots&\vdots&\vdots&\ddots&\vdots&\vdots \\
0&0&0&\cdots&1/N^2&0\\
0&0&0&\cdots&(N-1)\alpha +(N-1)^2/N^2&0\end{pmatrix}.
$$

\section{A look at some examples involving higher order operators}

 In section 2 we have shown that a scalar valued sequence of polynomials
   that are common eigenfunctions of a given differential operator gives rise,
   in a canonical fashion,
   to a sequence of matrix valued polynomials and to a differential operator
   that has them as common eigenfunctions.

   The purpose of this section is to illustrate a number of issues suggested
   by this general result in the case of some examples. These examples deal with
   what are called extended Krall-Laguerre and extended Krall-Jacobi polynomials
   in [GH].
   In both cases the scalar polynomials we consider satisfy a unique fourth
   order differential equation which therefore gives rise to a fourth order differential
   operator with matrix coefficients. The general theme of this section centers around
   the existence of other matrix valued differential operators having these same
   polynomials as common eigenfunctions very much in the spirit of [CG2].

\subsection{The scalar extended Krall-Laguerre polynomials}

The scalar valued (extended)  Krall--Laguerre polynomials $p_n(x)$ first introduced in \cite{GH} are given, for $n = 0,1,2,\dots$, by
\[
p_n(x) = \frac {1}{x} (q_{n+1}(x) + x_{n+1}q_n(x) + y_nq_{n-1}(x))
\]
where
\begin{align}\label{xey1}
x_n &= \alpha + \frac {2n^2+2(R-1)n-R}{(n-1)+R} \\
\label{xey2} y_n &= \frac {n(n+\alpha)(n+1+R)}{n+R}
\end{align}
and $q_n(x) = L^\alpha _n(x)$ stand for the familiar Laguerre polynomials given by
\[
L^\alpha _n(x) = (x-(2n-1+\alpha))L^\alpha _{n-1}(x) - (n-1)(n-1+\alpha)L^\alpha _{n-2}(x)
\]
and $L^\alpha _{-1}(x) = 0$, $L^\alpha _0(x) = 1$.

Here $R$ is a free parameter, and so is $\alpha$, which is taken to be larger than $-1$.

The $p_n(x)$ are obtained in \cite{GH} by an
application of the Darboux process starting from the Laguerre polynomials $L^\alpha _n(x)$ which are orthogonal and
satisfy the three term recursion relation given above. For generic values of $\alpha > -1$ the $p_n(x)$ do
not satisfy such a three term recursion relation, but they still satisfy
the five term recursion relation given by
\begin{equation}\label{ttrrKL}
\begin{pmatrix}
x_1 & 1 \\
y_1 & x_2 & 1 \\
& y_2 & x_3 &\cdot \\
& &\cdot &\cdot &\cdot \\
& & &\cdot &\cdot &\cdot \\
& & & &\cdot &\cdot &\cdot
\end{pmatrix} \quad \begin{pmatrix}
{\bar x}_1 & 1 \\
{\bar y}_1 & {\bar x}_2 & 1 \\
& {\bar y}_2 & {\bar x}_3 &\cdot \\
& &\cdot &\cdot &\cdot \\
& & &\cdot &\cdot &\cdot \\
& & & &\cdot &\cdot &\cdot
\end{pmatrix} \quad \begin{pmatrix}
p_0(x) \\
p_1(x) \\
p_2(x) \\
\cdot \\
\cdot \\
\cdot
\end{pmatrix} = x^2 \begin{pmatrix}
p_0(x) \\
p_1(x) \\
p_2(x) \\
\cdot \\
\cdot \\
\cdot
\end{pmatrix}
\end{equation}

where $x_n$ and $y_n$ are given above and ${\bar x}_n$ and ${\bar y}_n$ are given as follows
\begin{align}\label{xeybar1}
{\bar x}_n &= \alpha + \frac {2n^2 + 2(R-1)n-R}{n+R} \\
\label{xeybar2} {\bar y}_n &= \frac {n(n+\alpha)(n-1+R)}{n+R}
\end{align}
The point of the construction in \cite{GH} is that the new family of polynomials
$p_n(x)$ are common eigenfunctions of a fixed fourth order differential operator in the
spectral parameter $x$, as will be seen shortly.

 The method discussed in \cite{GH} shows that a proper use of the
Darboux process will take us from a trivial bispectral situation, i.e. the Laguerre polynomials
$L^\alpha _n(x)$ are eigenfunctions of a fixed second order differential operator and they are the common
eigenvectors of a tridiagonal difference operator, to a much less trivial one: the polynomials
$p_n(x)$ feature in a higher order bispectral situation that is not a trivial consequence of the
original one.

In our case we have
\begin{align}\label{fode}
[x^2D^4 - 2x(x-\alpha-2)D^3 &+ (x(x-2(R+\alpha)-6) + \alpha(\alpha+3))D^2 \\\nonumber
&+ 2((R+1)x - (\alpha+1)R-\alpha)D + R(R+1)]p_n(x) \\\nonumber
&= (R+n)(R+n+1)p_n(x).
\end{align}
We now produce out of these $p_n(x)$ a sequence of matrix valued polynomials of size $2 \times 2$
using the method explained in Section 2.

For concreteness we display below $P_0(x)$ and $P_1(x)$.  We have
\begin{align*}
P_{0}(x) &= \begin{pmatrix}
1 & 0 \\
-\frac {\alpha+(\alpha+1)R}{R+1} & 1
\end{pmatrix} \\
P_{1}(x) &= \begin{pmatrix}
x+\frac {(\alpha+2)((\alpha+1)R+2\alpha)}{R+2} & -\frac {2((\alpha+2)R+2\alpha+3)}{R+2} \\
-3x\frac {((\alpha+3)R+3\alpha+8)}{R+3} -\frac {(\alpha+2)(\alpha+3)((\alpha+1)R+3\alpha)}{R+3}
& x+\frac {3(\alpha+3)((\alpha+2)R+3\alpha+4)}{R+3}
\end{pmatrix}
\end{align*}

Here we display a fourth order differential operator $B$, arising by
applying the method explained in the previous section,  that has the matrix valued polynomials
$(P_{n}(x))_n$ as a set of common eigenfunctions

\begin{align*}
B &= D^4 x^3I + D^3 \begin{pmatrix}
(\alpha+5)x^2 &-x^2   \\
-x^3& (\alpha+7)x^2
\end{pmatrix}  + D^2 \begin{pmatrix}
\frac {x^2+ x(\alpha^2+9\alpha+15)}{4} & -\frac {(R+\alpha+6)x}{2}\\
-\frac {(R+\alpha+9)x^2}{2}  &\frac {x^2+(\alpha^2+15\alpha+39)x}{4}
\end{pmatrix} \\
&\quad \quad\quad +  D \begin{pmatrix}
\frac {(2R+3)x}{8}  + \frac {\alpha(\alpha+3)}{8} & -\frac {(\alpha+2)R+2\alpha+3}{4}\\
-x\frac {(\alpha+4)R+4(\alpha+3)}{4}&x\frac {2R+5}{8} + \frac {3(\alpha+1)(\alpha+4)}{8}
\end{pmatrix}  + \begin{pmatrix}
-\frac {R+1}{8} & 0\\
-\frac {(\alpha+1)R+\alpha}{8}  & 0
\end{pmatrix}.
\end{align*}

The matrix valued eigenvalue is read-off from the expression
\[
P_{n}(x) B= \left[ \begin{pmatrix}
-\frac {R+1}{8} & 0 \\
0 & 0
\end{pmatrix} + \begin{pmatrix}
\frac {2R+1}{8} & 0 \\
0 & \frac {2R+3}{8}
\end{pmatrix} n + \frac {n^2}{4} I\right] P_{n}(x)
\]

\bigskip

Notice that our differential operator acts on the right on its argument. This is also the reason why the eigenvalue matrix appears before the eigenfunction $P_n.$

\bigskip

Observe that the differential operator above depends both on $R$ and $\alpha$ while the eigenvalue
does not. If one were to consider, for a fixed pair $R, \alpha$ the family of matrix valued
polynomials, and then in the spirit of \cite{CG2} the full algebra we would
end up getting
isomorphic algebras of differential operators with matrix coefficients for different values of $\alpha$.
This follows from the fact that these algebras of differential operators are isomorphic to the
algebra of matrix valued eigenvalues.

\bigskip

\subsection{Matrix extended Krall-Laguerre and orthogonality}

The sequence of matrix polynomials $(P_{n})_n$ is not always orthogonal with respect to a weight
matrix of measures. This only happens if each polynomial $P_{n}$ can be normalized using a
nonsingular matrix: $Q_{n}=\Gamma _nP_{n }$ so that $(Q_{n })_n$ is
an orthonormal sequence with respect to a weight matrix. In that case,the sequence $(Q_{n})_n$
satisfies a three term recurrence relation of the form
$$
xQ_{n}=A_{n+1}Q_{n+1}+B_{n}Q_{n}+A_{n}^*Q_{n-1},
$$
with $A_{n}$ non singular and $B_{n}$ Hermitian.

If we were dealing with a monic sequence of polynomials $P_n$ satisfying a three
term recursion, such as

$$
xP_{n}=P_{n+1}+B_{n}P_{n}+A_{n}P_{n-1},
$$

this symmetrization can be done only when one can find positive definite matrices $S_n$ such that $B_n S_n$ is Hermitian and $A_n S_{n+1}=S_n$.

\bigskip

As we noticed at the end of section 2, given a scalar family of polynomials $p_n(x)$
we could attempt a scalar symmetrization of the (higher order) recursion relation.
Given the Krall-Laguerre polynomials $(p_{n})_n$ we can look for a normalization
$r_{n}(x)=\tau_np_{n}(x)$ satisfying a
symmetric five term recurrence relation of the form
\begin{equation}\label{ftrr}
x^2r_{n}=a_{n+2}r_{n+2}+b_{n+1}r_{n+1}+
c_{n}r_{n}+b_{n}r_{n-1}+a_{n}r_{n-2},
\end{equation}
where $a_{n}\not =0$, $n\ge 0$.

If we use the parameters
$(x_n)_n$, $(y_n)_n$, $(\bar x_n)_n$ and $(\bar y_n)_n$ (see
(\ref{xey1}), (\ref{xey2}), (\ref{xeybar1}) and (\ref{xeybar2})) it turns out that we would need

$$
y_{n+1}\bar y_n=
\frac{(y_n\bar x_n+x_{n+1}\bar y_n)(y_{n+1}\bar x_{n+1}+x_{n+2}\bar y_{n+1})}{((x_n+\bar x_{n+1})(x_{n+1}+\bar x_{n+2})}.
$$
This only happens when $\alpha =0$ and then:
\begin{align*}
a_{n+2,R,0}&=\displaystyle (n+1)(n+2)\sqrt {\frac{(n+R)(n+R+3)}{(n+R+1)(n+R+2)}},\\&\\
b_{n+1,R,0}&=\displaystyle 2(n+1)\frac{2(n+1)R^2+(2n+1)(2n+3)R+2n(n+1)(n+2)}{(n+R+1)\sqrt{(n+R)(n+R+2)}},\\&\\
c_{n,R,0}&=\displaystyle \frac{n^2(n+R+1)^2+(n+1)^2(n+R)^2+((2n(n+1)+(2n+1)R)^2}{(n+R)(n+R+1)}.
\end{align*}
For $\alpha =0$, the five terms recurrence relation (\ref{ftrr}) is just an iteration of
the three term recurrence relation for the Laguerre type orthogonal polynomials
$(L_n^{R})_n$ introduced by Koornwinder (see \cite{Koo} and also \cite{KoKo}). They are orthogonal
in $[0,+\infty )$ with respect to the measure $e^{-x}+\frac{1}{R}\delta_0$. The four order
differential equation (\ref{fode}) reduces to the well-known four order differential equation for this family
(see, for instance, \cite{LK}, p.112).

The corresponding matrix polynomials $(P_{n,R,0})_n$ are then orthogonal in $[0,+\infty)$
with respect to the weight matrix
$$
e^{-\sqrt x}\begin{pmatrix}\displaystyle \frac{1}{\sqrt x}+\frac{2}{R}\delta_0 &1\\
1&\sqrt x
\end{pmatrix} .
$$

\bigskip

In this case one can see that the more general attempt to symmetrize the matrix
valued recursion relation fails too, except in the case $\alpha=0$.

\bigskip

\subsection{The scalar extended Krall-Jacobi case}

The scalar valued (extended) Krall--Jacobi polynomials $p_n(x)$,
first introduced in [GH] are given by
$$
q_n(x) = \frac {1}{x+1} (y_np_{n-1}(x) + x_{n+1}p_n(x) + p_{n+1}(x))
$$
where, as in $(4.11)$ and $(4.12)$ in [GH], we have
\begin{align*}
x_n &= \frac {2(\alpha+\beta+n)(\alpha\beta + \beta^2 + 2n\alpha +
2(n-1)\beta +
2n(n-1))}{(\alpha+\beta+n-1)(\alpha+\beta+2n-2)(\alpha+\beta+2n)} \\
& \quad\quad \quad\quad - \frac
{2(\alpha+n-1)R}{(\alpha+\beta+n-1)(\alpha+\beta+2n-2)\theta_{n-1}} \\
y_n &= a_n \frac {\theta_{n+1}}{\theta_n}
\end{align*}
and
$$
\theta_n = n^2 + (\alpha+\beta)n + R.
$$
Here $p_n(x)$ stands for the usual (monic) Jacobi polynomials satisfying
$$
xp_n(x) = a_n p_{n-1}(x) + b_{n+1}p_n(x) + p_{n+1}(x)
$$
with
\begin{align*}
a_n &= \frac
{4n(n+\alpha+\beta)(n+\alpha)(n+\beta)}{(2n+\alpha+\beta-1)(2n+\alpha+\beta)^2(2n+\alpha+\beta+1)}
\\
b_n &= \frac {\beta^2-\alpha^2}{(2n+\alpha+\beta-2)(2n+\alpha+\beta)}.
\end{align*}

As can be seen in \cite{GH}, the crux of the construction of these polynomials is that the polynomials $p_n(x)$ and $q_n(x)$ are also related by means of

$$
p_n(x) = \frac {1}{x+1} (\bar y_nq_{n-1}(x) + \bar x_{n+1}q_n(x) + q_{n+1}(x))
$$

As explained in \cite{GH} the expressions for $\bar x_n$ and $\bar y_n$ can be obtained by
replacing $b_n$ by $b_n +1$ in (3.17) of that paper. The resulting expression for
$\bar x_n$ is correctly given by the first part of (4.13), but there is an error in the
expression for $\bar y_n$ reported in \cite{GH}, and we give the correct expression
below, namely

$$\bar y_n=x_{n+1}(x_{n+2}-b_{n+1}-b_{n+2}-2)-y_{n+1}+a_n+a_{n+1}+(b_{n+1}+1)^2$$

For completness we reproduce the expression for $\bar x_n$ from \cite{GH}, namely

\begin{align*}
\bar x_n &= \frac {2(\alpha+\beta+n-1)(\alpha\beta + \beta^2 + 2(n-1)\alpha +
2n\beta +
2n(n-1))}{(\alpha+\beta+n)(\alpha+\beta+2n-2)(\alpha+\beta+2n)} \\
& \quad\quad \quad\quad + \frac
{2(\alpha+n)R}{(\alpha+\beta+n)(\alpha+\beta+2n)\theta_{n}} 
\end{align*}

Finally $R$ is a free parameter. 

\bigskip

These polynomials $q_n(x)$ satisfy,
for generic values of $\alpha,\beta$, a five term recursion relation
and they are the common eigenfunctions of a fourth order differential
operator $PQ$ with
$$
P = (x^2-1)D^2 + ((\alpha+\beta+1)x-\beta+\alpha+1)D+R
$$
and
$$
Q = (x^2-1)D^2 + ((\alpha+\beta+3)x-\beta+\alpha-1)D + (R+\alpha+\beta+1).
$$
More explicitly we have
$$
PQ q_n(x) = \lambda_n q_n(x)
$$
with
$$
\lambda_n = (R+n(\alpha+\beta+n))(R+(n+1)(\alpha+\beta+n+1)).
$$

\bigskip

The five term recursion relation expresses the product $$(x+1)^2 q_n(x)$$ as
a linear combination of the polynomials $q_{n+2},q_{n+1},q_n,q_{n-1},q_{n-2}$.

\bigskip

One could raise again the possibility that this five recursion could be made
symmetric by premultiplying the scalar polynomials by appropriate $n$ dependent 
constants. This would require, as in the case of the modified Krall-Laguerre 
polynomials that a certain identity should hold among the quantities $x_n,y_n,\bar x_n,\bar y_n$. 
An explicit computation using the expressions above shows that this happens only
for $\beta=0.$

If one obtains the matrix valued polynomials out of these scalar ones and tries
to symmetrize the resulting three term recursion relation one finds, once again, that
in this case nothing is gained by allowing this more general cure to our problem:
there is a positive definite orthogonality matrix only in the case of $\beta=0.$

\bigskip

We consider now (as an example) the special case given by $\alpha = 3/2$, $\beta =
7/8$, $R = 7$.  In this case the scalar fourth order operator
mentioned earlier becomes
$$
(x^2-1)^2D^4 + \frac {(x^2-1)(51x+5)}{4} D^3 + \frac {7}{64} (503x^2
+ 98x - 281)D^2 + \frac {35}{32} (83x + 21)D.
$$

The general method indicated  earlier produces a fourth order
differential operator with matrix coefficients satisfied by the
matrix valued polynomials obtained from the sequence $p_n(x)$.  In
the spirit of [CG2] one can study the algebra of all differential
operators associated with this family of matrix valued polynomials.

In this example we have checked that the algebra contains only one
(up to scalars) differential operator of order four.  If we look for
operators of order less than or equal to six, we find only one such
operator (modulo operators of order $4$), whose highest order terms
are
\begin{align*}
D^6 &\begin{pmatrix}
(x-1)^3x^3 & 0 \\
0 & (x-1)^3x^3
\end{pmatrix} \\ &\quad \quad+ \frac {3}{16} D^5 \begin{pmatrix}
(x-1)^2x^2(107x-40) & 5(x-1)^2x^2 \\
5(x-1)^2x^3 & (x-1)^2x^2(123x-56)
\end{pmatrix} + \dots
\end{align*}
In particular the algebra cannot be generated by the fourth order
operator constructed out of the scalar one.
As mentioned in the introduction, this also gives examples of situations where the
lowest order of a (nontrivial) differential operator in the algebra is higher than
two.

Extensive computations suggest that these two observations hold for
generic values of $\alpha,\beta,R$. Just as in a previous example, and all the
examples below, all differential operators with matrix valued coefficients act on the right of their matrix valued arguments.

\bigskip

\subsection{Lower order operators connected to the extended Krall-Jacobi case}

A remarkable phenomenon takes place when $\alpha$ and $\beta$ differ
by one.  Consider the case $\beta = \alpha + 1$.  Here, very much as
in the generic case, we find that the scalar extended Krall--Jacobi
polynomials are the common eigenfunctions of only one fourth order
differential operator, namely
\begin{align*}
(x^2-1)^2D^4& + 2(1-x^2)((2\alpha+5)x-1)D^3 \\
&\quad\quad-\ 2(x((2\alpha^2+11\alpha+19)x-2\alpha-2)-3\alpha - 13)D^2 \\
&\quad\quad\quad\quad-\ 2((2\alpha+3)(2\alpha+9)x-7)D.
\end{align*}
A surprise comes about when we build the matrix valued polynomials
going along with the scalar ones alluded to above.

\bigskip

 We consider the
differential operators with matrix coefficients and order not
exceeding three.  Notice that for generic values of $\alpha,\beta$
this gives only the identity operator.  The situation now is entirely
different.  We find four nontrivial operators, which along with
identity, give a basis for this five dimensional space.

The matrix valued polynomials $P_n(x)$ are common eigenfunctions of
two linearly independent third order operators.  But there are other
surprises in store.  They also satisfy differential equations of orders one and two,
namely, recalling that $\beta=\alpha+1$,
$$
P'_n(x) \begin{pmatrix}
0 & x-1 \\
0 & 0
\end{pmatrix} + P_n(x) \begin{pmatrix}
1 + \frac {R}{2\beta} & \beta - \frac {R}{2\beta} \\
0 & 0
\end{pmatrix} = \left[ \begin{pmatrix}
1 + \frac {R}{2\beta} & \beta - \frac {R}{2\beta} \\
0 & 0
\end{pmatrix} + n \begin{pmatrix}
0 & 1 \\
0 & 0
\end{pmatrix} \right] P_n
$$
as well as
\begin{align*}
P''_n(x) \begin{pmatrix}
0 & x(x-1) \\
0 & x(x-1)
\end{pmatrix} &+ P'_n(x) \begin{pmatrix}
0 & x+\alpha \\
0 & \frac {(2\alpha+5)x-3}{2}
\end{pmatrix} + P_n(x)A_0 \\
&= \left[ A_0 + n \begin{pmatrix}
0 & 0 \\
0 & \frac {2\alpha+3}{2}
\end{pmatrix} + n^2 \begin{pmatrix}
0 & 1 \\
0 & 1
\end{pmatrix} \right] P_n(x)
\end{align*}
where
$$
A_0 = \begin{pmatrix}
-\frac {4\alpha+5}{2} & -(\alpha+1)^2 \\
0 & 0
\end{pmatrix} + \frac {R(3\alpha+4)}{4(\alpha+1)} \begin{pmatrix}
-1 & 1 \\
0 & 0
\end{pmatrix}.
$$

   \end{document}